\documentclass[a4paper,10pt]{article}

\usepackage[dvipdfmx]{graphicx} 
\usepackage{amscd,amsmath,amssymb,amsthm}
\usepackage{float}
\usepackage{url}

\usepackage{stmaryrd}

\newtheorem{theorem}{Theorem}
\newtheorem{proposition}{Proposition}
\newtheorem{lemma}{Lemma}
\newtheorem{corollary}{Corollary}

\theoremstyle{definition}
\newtheorem{definition}{Definition}
\newtheorem{example}{Example}

\begin{document}

\pagestyle{plain}

\title{\textsc{Labeled Fibonacci Trees}}

\author{
St\'{e}phane Legendre\\
\\
\textit{Team of Mathematical Eco-Evolution}\\
\textit{Ecole Normale Sup\'{e}rieure}\\
\textit{Paris, France}\\}

\date{}

\maketitle

The study describes a class of integer labelings of the Fibonacci tree, the tree of descent introduced by Fibonacci. In these labelings, Fibonacci sequences appear along ascending branches of the tree, and it is shown that the labels at any level are consecutive integers. The set of labeled trees is a commutative group isomorphic to $\mathbb{Z}^2$, and is endowed with an order relation. Properties of the Wythoff array are recovered as a special instance, and further properties of the labeled Fibonacci trees are described. These trees can be viewed as generalizations of the Wythoff array.

\section{Introduction}
The book of Hofstadter \cite{Hofstadter 1980} (Figure 30, p. 136) contains an outstanding mathematical object, a subtree of the Fibonacci tree labeled by the set of integers (highlighted on the right of Fig. \ref{figure:fibtree:F12}). In the book, the tree is used to represent the values taken by the recursive function $g(n) = n - gg(n-1)$. From the root of the tree, level after level, consecutive integers miraculously match Fibonacci sequences appearing along ascending branches of the tree: the Fibonacci sequence on the main branch, the Lucas sequence on the second branch, and other Fibonacci-type sequences. This correspondence was proved by Tognetti, Winley and van Ravenstein \cite{Tognetti 1990} in 1990.

In fact, \textit{all} positive Fibonacci sequences eventually appear as ascending branches of the tree. For example, the Lucas sequence
\begin{displaymath}
\ldots -4, \; 3, \; -1, \; 2, \; 1, \; 3, \; \underline{4}, \; \underline{7}, \; \underline{11}, \; \underline{18}, \; \underline{29}, \; \ldots
\end{displaymath}
is represented in the tree from the underlined terms. This is what I could show when I discovered the tree in 1986. To some disappointment, I realized that a similar result had already been found by Morrison \cite{Morrison 1980} in the context of the Wythoff array. Like the Hofstadter tree, the Wythoff array \cite{SloaneW} contains every integer exactly once, and represents every Fibonacci sequence exactly once.\\

In this study, a set of labeled Fibonacci trees is described, generalizing the  Hofstadter tree. First, the (unlabeled) Fibonacci tree is introduced (Section \ref{section:tree}), and properties of the Fibonacci words and Wythoff pairs are recalled (Section \ref{section:word}). Then labeling rules for the Fibonacci tree are given (Section \ref{section:labels}). According to these rules, Fibonacci sequences appear as successive labels along ascending branches of the tree. It is shown that the labels at any level of the tree form consecutive integers. In Section \ref{section:wythoff}, the Hofstadter tree and the Wythoff array are recovered as a special instance of these labeled trees. In Section \ref{section:set}, it is shown that the set $\Phi$ of labeled Fibonacci trees has the structure of a commutative group isomorphic to $\mathbb{Z}^2$. The representation of integer intervals and Fibonacci sequences by elements of $\Phi$ are explored in Section \ref{section:representation}. Finally, an order relation on the set $\Phi$ is described in Section \ref{section:order}. According to this relation, only two trees contain nested copies of themselves. They correspond to the Wythoff arrays representing the positive and negative Fibonacci sequences.

\section{The Fibonacci tree}\label{section:tree}

The \textit{Fibonacci tree} is the tree of descent of the rabbit family introduced by Leonardo da Pisa in his book \textit{Liber Abaci} (1202). In the original problem (slightly reformulated), at each time step:

\begin{itemize}
\item An adult female $u$ survives to the next generation, and gives birth to a female juvenile $v$.
\item A juvenile $v$ survives to the next generation, and becomes an adult $u$.
\end{itemize}

\noindent These rules translate
\begin{equation}\label{fibonacci:scheme}
\left\{
\begin{array}{c c c}
u &\rightarrow &uv\\
v &\rightarrow &u
\end{array}\right\}.
\end{equation}

\begin{figure}[t]
\centering
\includegraphics[scale=0.4]{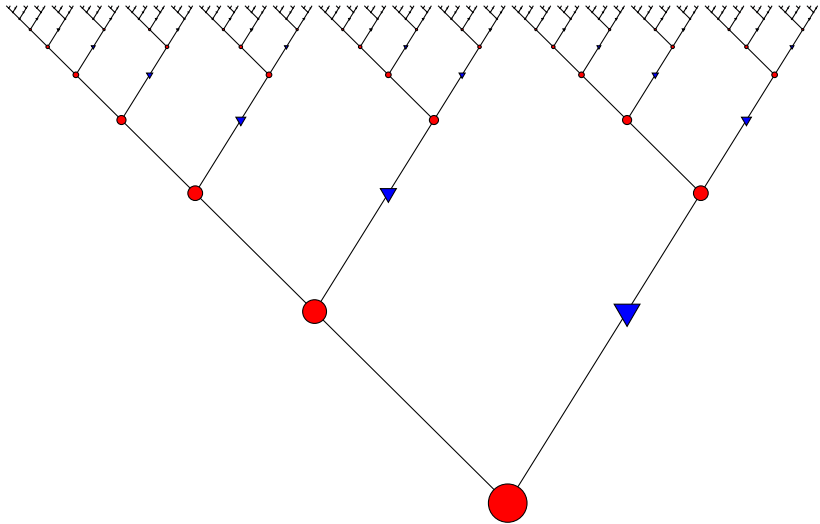}
\caption{The Fibonacci tree.}
\label{figure:fibtree}
\end{figure}

\begin{figure}[h]
\centering
\includegraphics[scale=0.4]{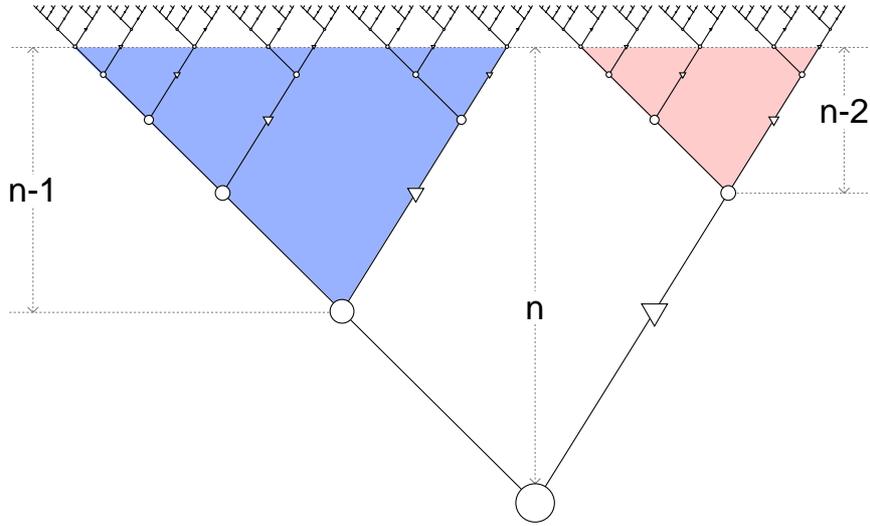}
\caption{The Fibonacci recursion.}
\label{figure:fibtree:recursion}
\end{figure}

Starting from a single adult $u$ at time $n=0$, the tree is built according to scheme (\ref{fibonacci:scheme}). It is drawn using the golden ratio in Figure \ref{figure:fibtree} with $u$-nodes representing adults (circles) and $v$-nodes representing juveniles (triangles). The tree is assumed infinite, and because of the recursive nature of scheme (\ref{fibonacci:scheme}), it contains an infinity of nested copies of itself.

The root has level $n = 0$. The population size at time $n$ is the number $G_n$ of nodes at level $n$. The sequence $\mathbf{G}$ satisfies the Fibonacci recursion
\begin{displaymath}
G_{n} = G_{n-1} + G_{n-2}
\end{displaymath}
with $G_{0}  =1$, $G_{1} = 2$. A visual proof is given in Figure \ref{figure:fibtree:recursion}.
By the definition of the Fibonacci sequence,
\begin{displaymath}
\mathbf{F}:0,1,1,2,3,5,8,13,21,34,55,\ldots
\end{displaymath}
we have $G_n = F_{n+2}$.

Any sequence $\mathbf{G}$ satisfying the Fibonacci recursion can be extended to $\mathbb{Z}$ by setting $G_{n} = G_{n+2} - G_{n+1}$ for $n < 0$. Two such Fibonacci sequences $\mathbf{G}$ and $\mathbf{G'}$ are \textit{equivalent}, $\mathbf{G} \sim \mathbf{G'}$, if they are identical up to a shift of index. 

For $(a,b) \in \mathbb{Z}^2$, $\mathbf{F}^{a,b}$ denotes the Fibonacci sequence whose terms of index 0 and 1 are $a$ and $b$ respectively. Then,
\begin{equation}\label{equation:Fab}
F^{a,b}_{n} = aF_{n-1} + bF_{n}, \quad n \in \mathbb{Z}.
\end{equation}

\section{Fibonacci words and Wythoff pairs}\label{section:word}

In this section classical and less classical results about Fibonacci words and Wythoff pairs are recalled, to be used in the sequel.

\textit{Fibonacci words} $W_{n}$ over the alphabet $\{{u,v}\}$ are generated from the word $W_{0}=u$ by the substitutions of scheme (\ref{fibonacci:scheme}):
\begin{displaymath}
\begin{array}{l l l}
W_{0} &= &u\\
W_{1} &= &uv\\
W_{2} &= &uvu\\
W_{3} &= &uvuuv\\
W_{4} &= &uvuuvuvu\\
W_{5} &= &uvuuvuvuuvuuv\\
{}    &\ldots
\end{array}
\end{displaymath}
Fibonacci words satisfy the Fibonacci recursion
\begin{displaymath}
W_{n} = W_{n-1}W_{n-2}.
\end{displaymath}
The length of $W_{n}$ is 
\begin{displaymath}
|W_n| = F_{n+2}, 
\end{displaymath}
with $F_{n+1}$ letters $u$ and $F_n$ letters $v$.

\begin{proposition}\label{proposition:fibonacci:tree:structure}
At level $n$ of the Fibonacci tree, the pattern of $u$-nodes and $v$-nodes is given by the Fibonacci word $W_{n}$.
\end{proposition}

\begin{proof}
The Fibonacci word is generated by the same scheme as the Fibonacci tree with the same initial condition.
\end{proof}

The \textit{Wythoff pairs} $(u(n),v(n))$, are given by two complementary sequences $\mathbf{u}$ and $\mathbf{v}$ over the set $\mathbb{N}^*$ of positive integers (sequences A000201 and A001950 in the Online Encyclopedia of Integer Sequences \cite{Sloane}), with $u(1) = 1$ and
\begin{equation}\label{equation:v}
v(n) = u(n) + n.
\end{equation}
The Wythoff pairs can be extended to $\mathbb{Z}$, by setting $u(0) = v(0) = -1$, and $u(-n) = -u(n)-1$, $v(-n) = -v(n)-1$ for $n \in \mathbb{N}^*$ (Table \ref{table:wythoff:pairs}). Then (\ref{equation:v}) still holds. The following formula, valid for $n \in \mathbb{N}^*$, also holds for $n \in \mathbb{Z}$:
\begin{equation}\label{equation:uu}
v(n) = uu(n) + 1.
\end{equation}
Moreover, for any nonzero integer $n \in \mathbb{Z}^*$,  
\begin{equation}\label{equation:beatty}
u(n) = \lfloor n\varphi \rfloor, \quad \varphi = \frac{1 + \sqrt{5}}{2} \text{  the golden ratio}.
\end{equation}

\begin{table}[h]
\centering
\begin{tabular}{l | c c c c c c c c c c c c c c c c}
{$n$}        &-6  &-5  &-4  &-3 &-2 &-1 &0  &1 &2 &3 &4  &5  &6  &7  &8\\
\hline
$\mathbf{u}$ &-10 &-9  &-7  &-5 &-4 &-2 &-1 &1 &3 &4 &6  &8  &9  &11 &12\\
$\mathbf{v}$ &-16 &-14 &-11 &-8 &-6 &-3 &-1 &2 &5 &7 &10 &13 &15 &18 &20\\
\hline
\end{tabular}
\caption{The Wythoff pairs.}
\label{table:wythoff:pairs}
\end{table}

Positive Wythoff pairs form pairs of consecutive terms of positive Fibonacci sequences. A Wythoff pair is \textit{primitive} if its rank is a term of $\mathbf{u}$ \cite{Silber 1977}. It is written $(uu(j),vu(j))$, $j \in \mathbb{N}^*$. A primitive pair starts a novel Fibonacci sequence, one that has not yet appeared in the Wythoff pairs. 

This property extends to negative Wythoff pairs and negative Fibonacci sequences, when heading to the left from $n = 0$ in Table \ref{table:wythoff:pairs}, with an exception: the pair $(-2,-3)$ at rank $-1 = u(0)$ is $\textit{not}$ primitive despite its rank being a term of $\mathbf{u}$. Indeed, this pair corresponds to the sequence $-1, -1, -2, -3, \ldots$ whose first terms $(-1,-1)$ appear at rank $0$. 

To summarize, the primitive Wythoff pairs are $(uu(j),vu(j))$ for $j \in \mathbb{Z}^*$.

\begin{proposition}\label{proposition:fibonacci:word:u}
Consider the word $W_n$ generated from $W_{n-1}$ by the substitutions (\ref{fibonacci:scheme}). In $W_n$,  the position $i \geqslant 1$ of a letter is given by the corresponding Wythoff sequence: $i = u(k)$ if the letter is $u$, where $k$ is the number of letters $u$ up to position $i$, and $i = v(l)$ if the letter is $v$, where $l$ is the number of letters $u$ up to the letter $u$ in $W_{n-1}$ that generates this letter $v$ in $W_n$.
\end{proposition}
\begin{proof}
Letters $v$ in $W_n$ come uniquely from letters $u$ in $W_{n-1}$ by the substitution $u \rightarrow uv$, so that the $l$-th occurence of $u$ in $W_{n-1}$, in position $i_l$, produces the $l$-th occurence of $v$ in $W_n$, in position $j_l$. Up to position $i_l$, there are $i_l - l$ occurences of $v$, each producing $u$ by the substitution $v \rightarrow u$. The number of letters produced by the two substitutions up to the $l$-th occurence of $v$ is thus $2l + i_l - l = i_l + l = j_l$. The position of the first $u$ is $1 = u(1)$, and since $(i_l,j_l)$ form complementary sequences, we obtain: $i_l = u(l)$, and by (\ref{equation:v}), $j_l = u(l) + l = v(l)$. The position of $u$ in $W_{n-1}$ is $i_l = u(l)$. Similarly, the position of $u$ in $W_n$ is $i_k = u(k)$ where $k$ is the number of letters $u$ up to the given letter $u$ (included).
\end{proof}

\begin{proposition}\label{proposition:fibonacci:word:parent}
Consider the word $W_n$ generated from $W_{n-1}$ by the substitutions (\ref{fibonacci:scheme}), and a letter $y$ of $W_n$ that has $k$ letters $u$ on its left. In $W_{n-1}$, the position of the parent letter $x$ that generates $y$ is $k$.
\end{proposition}
\begin{proof}
The number $k$ of letters $u$ to the left of $y$ is equal to the number of letters $u$ to the left of $x$ in $W_{n-1}$ (by $u \rightarrow uv$) plus the number of letters $v$ to the left of $x$ in $W_{n-1}$ (by $v \rightarrow u$): this is exactly the position of $x$ in $W_{n-1}$.
\end{proof}

The following two technical lemmas will be used in Section \ref{section:representation}.

\begin{lemma}\label{lemma:uG}
Let $\mathbf{G}$ be a Fibonacci sequence. For any $i \in \mathbb{Z}^*$ there exists $n_1 \in \mathbb{N}$ such that $n \geqslant n_1$ implies $u(i + G_{n}) = u(i) + G_{n+1}$.
\end{lemma}
\begin{proof}
Let $\Delta_n = G_{n}\varphi - G_{n+1}$.
By (\ref{equation:beatty}), 
\begin{displaymath}
u(i + G_{n}) - G_{n+1} = \lfloor i\varphi + G_{n}\varphi \rfloor - G_{n+1} = \lfloor i\varphi + G_{n}\varphi - G_{n+1} \rfloor = \lfloor i\varphi + \Delta_n \rfloor. 
\end{displaymath}
We know that $\Delta_{n} = -\frac{1}{\varphi}\Delta_{n-1}$, so that $\Delta_n \rightarrow 0$ with alternating sign: $\epsilon > 0$ being given, there exists $n_1$ such that $n \geqslant n_1$ implies $-\epsilon < \Delta_n < \epsilon$. Then $i\varphi - \epsilon < i\varphi + \Delta_n < i\varphi + \epsilon$. Choosing $\epsilon = \inf(i\varphi - \lfloor i\varphi \rfloor, \lfloor i\varphi \rfloor + 1 - i\varphi)$, gives $\lfloor i\varphi \rfloor < i\varphi + \Delta_n < \lfloor i\varphi \rfloor + 1$. We obtain $u(i + G_{n}) - G_{n+1} = \lfloor i\varphi \rfloor = u(i)$.
\end{proof}

Before proceeding to the next lemma, let us recall the analysis of Brother U. Alfred \cite{Brother U Alfred 1963}. Any nonzero Fibonacci sequence $\mathbf{G}$ has two parts: the monotonic part going to the right, where the terms are of constant sign, and the alternating part on the left where the signs alternate. For a positive sequence, the separation between the parts occurs at the place where consecutive terms are $d-c$, $c$, $d$ with $d-c > c$ and $c < d$: $c$ is the smallest nonnegative term of the sequence, and the term previous to $d-c$ is negative. For a negative sequence, $c$ is the largest nonpositive term of the sequence, $d-c < c$, $c > d$, and the term previous to $d-c$ is positive. Let us call the rank $\nu$ of $c$ the \textit{reference index} of the sequence. 

To summarize, the reference index $\nu = \nu(\mathbf{G})$ is such that when $\mathbf{G}$ is positive, $G_{\nu-1} > G_{\nu} \geqslant 0$, and when $\mathbf{G}$ is negative, $G_{\nu-1} < G_{\nu} \leqslant 0$.

\begin{lemma}\label{lemma:uG0}
Let $\mathbf{G}$ be a nonzero Fibonacci sequence with reference index $\nu$. Then there exists a unique $n \geqslant \nu$ such that $u(i) + G_{n+1} = 0$ for $i = 1 - G_n$.
\end{lemma}
\begin{proof}
Let $\Delta_n = G_{n+1} - G_n\varphi$. For $i = 1 - G_n$, we write 
\begin{displaymath}
u(i) + G_{n+1} = \lfloor i\varphi + G_{n+1} \rfloor = \lfloor \varphi - G_n\varphi + G_{n+1} \rfloor = \lfloor \varphi + \Delta_n \rfloor. 
\end{displaymath}
As $\Delta_n \rightarrow 0$ and is bracketed by bounds of disjoint intervals, there exists a unique $n$ such that $-\varphi < \Delta_{n} < -\frac{1}{\varphi}$. Then $0 < \varphi + \Delta_{n} < 1$, and we obtain $u(i) + G_{n+1} = 0$. The conditions $-\varphi < \Delta_{n} < -\frac{1}{\varphi}$ and $u(i) + G_{n+1} = 0$ are equivalent. To prove that $n \geqslant \nu$, we have to check that for $m <\nu$, $\Delta_m$ is outside the appropriate bounds. As $|\Delta_m|$ increases with decreasing $m$, we need only consider the case $m = \nu-1$. For $\mathbf{G}$ positive, by the property of $\nu$, we have $G_{\nu-1} \geqslant G_{\nu} + 1$, and
\begin{displaymath}
\Delta_{\nu-1} = G_{\nu} - G_{\nu-1}\varphi \leqslant G_{\nu} - G_{\nu}\varphi - \varphi = -\frac{G_{\nu}}{\varphi} - \varphi \leqslant -\varphi.
\end{displaymath}
For $\mathbf{G}$ negative, a similar analysis shows that $\Delta_{\nu-1} \geqslant \varphi$.
\end{proof}

\section{Integer labeling of the Fibonacci tree}\label{section:labels}

For $(a,b) \in \mathbb{Z}^2$, the \textit{labeled Fibonacci tree} $\mathcal{F}^{a,b}$ is the Fibonacci tree (Fig. \ref{figure:fibtree}) with $u$-nodes and $v$-nodes labeled according to the following rules:

\begin{itemize}
\item The root is a $u$-node labeled $a$. Its child nodes are a $u$-node labeled $b-1$ and a $v$-node labeled $b$. 
\item The child nodes of a $u$-node labeled $y$, whose parent node is labeled $x$, are a $u$-node labeled $x+y-1$ and a $v$-node labeled $x+y$.
\item The child node of a $v$-node labeled $t$, whose parent node is labeled $z$, is a $u$-node labeled $z+t$. 
\end{itemize}
The labeled Fibonacci tree $\mathcal{F}^{0,1}$ is displayed in Figure \ref{figure:fibtree:F01}. The labels are read from the root and from left to right. By construction,
\begin{itemize}
\item The labels of child nodes of $u$-nodes are consecutive integers.
\item Ascending branches of the tree, where $u$-nodes and $v$-nodes alternate, are labeled according to the Fibonacci recursion.
\end{itemize}

\begin{figure}[ht]
\centering
\includegraphics[scale = 0.44]{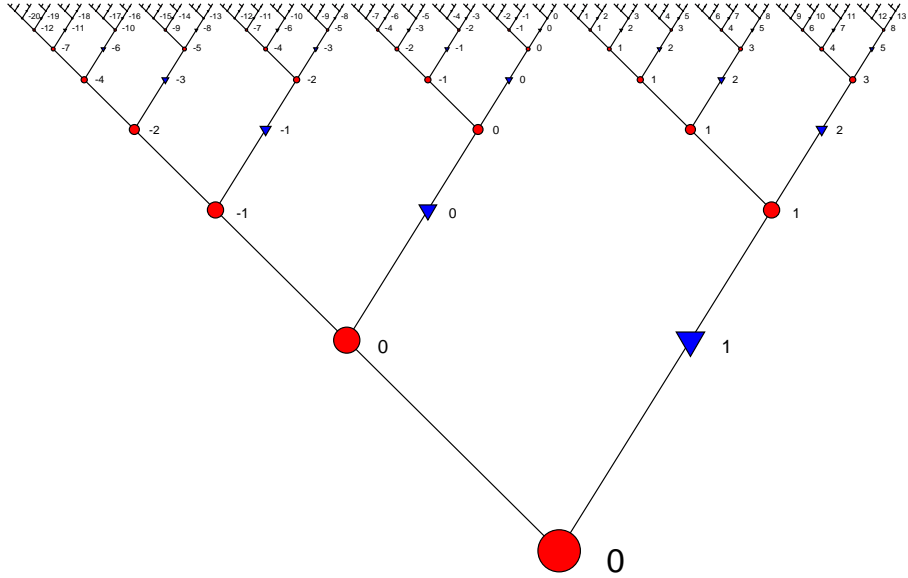}
\caption{The labeled Fibonacci tree $\mathcal{F}^{0,1}$.}
\label{figure:fibtree:F01}
\end{figure}

At level $n$, a labeled Fibonacci tree has $F_{n+2}$ nodes. Let $A_n$ and $B_n$ be the leftmost and rightmost labels at level $n$. From the labeling rules, we have the sequence $\mathbf{B}: a, b, a+b, \ldots$, so that
\begin{displaymath}
B_n = F^{a,b}_n.
\end{displaymath}
The leftmost node is a $u$-node, and its label $A_n$ is given by the sequence $\mathbf{A}: a, b-1, a+b-2, \ldots$, with
\begin{displaymath}
A_n = A_{n-1} + A_{n-2} - 1, \text{ for } n \ge 2.
\end{displaymath}
Using this expression, it is checked by induction that
\begin{displaymath}
A_n = F^{a,b}_n - F_{n+2} + 1.
\end{displaymath}

For $x,y \in \mathbb{Z}$, $x \leqslant y$, $\llbracket x \cdots y \rrbracket$ denotes the interval of consecutive integers from $x$ to $y$, and we use the notation
\begin{displaymath}
\llbracket x+z \cdots y+z \rrbracket = \llbracket x \cdots y \rrbracket + z.
\end{displaymath}
It is possible to add or substract two intervals of the same length:
\begin{displaymath}
\begin{array}{c c c}
\llbracket x \cdots y \rrbracket + \llbracket x' \cdots y' \rrbracket &=& \llbracket x+x' \cdots y+y' \rrbracket,\\
\llbracket x \cdots y \rrbracket - \llbracket x' \cdots y' \rrbracket &=& \llbracket x-x' \cdots y-y' \rrbracket.
\end{array}
\end{displaymath}

We denote
\begin{displaymath}
\mathcal{L}^{a,b}_n = \llbracket A_n \cdots B_n \rrbracket = \llbracket F^{a,b}_n - F_{n+2} + 1 \cdots F^{a,b}_n \rrbracket = \llbracket -F_{n+2} + 1 \cdots 0 \rrbracket + F^{a,b}_n. 
\end{displaymath}
As $F^{0,0}_n = 0$ for all $n$, we have 
\begin{displaymath}
\mathcal{L}^{0,0}_n = \llbracket -F_{n+2} + 1 \cdots 0 \rrbracket,
\end{displaymath}
so that
\begin{displaymath}
\mathcal{L}^{a,b}_n = \mathcal{L}^{0,0}_n + F^{a,b}_n. 
\end{displaymath}

\begin{theorem}\label{theorem:label:Fab}
At level $n$ of the tree $\mathcal{F}^{a,b}$, the labeling is made of consecutive integers in the interval $\llbracket F^{a,b}_n - F_{n+2} + 1 \cdots F^{a,b}_n \rrbracket$.
\end{theorem}
\begin{proof}
We use induction on $n$. The result is true for $ n \le 2$, and for $n>2$, let us consider a node $Q$ labeled $y$ at level $n-1$, and that is not a rightmost node. By induction hypothesis, the labelings at levels $n-1$ and $n-2$ are made of consecutive integers. Hence, the node $Q'$ next to $Q$ at level $n-1$ has label $y+1$. Let us assume that the parent node of $Q$, $P$ at level $n-2$, has label $x$, and that the parent node of $Q'$ is $P'$. We consider all the different cases implied by the topology of the tree:
\begin{enumerate}
\item $Q$ is a $u$-node and $Q'$ is a $u$-node. The only configuration is that the parent node $P$ is a $v$-node and $P'$ is a $u$-node next to $P$, that has label $x+1$. By the labeling rules, the two child nodes of $Q$, at level $n$, have labels $(x+y-1,x+y)$, and two child nodes of $Q'$ have labels $((x+1) + (y+1)-1,(x+1)+(y+1)) = (x+y+1,x+y+2)$.
\item $Q$ is a $u$-node and $Q'$ is a $v$-node. Then $Q$ and $Q'$ have the same parent node $P$, a $u$-node labeled $x$. At level $n$, the two child nodes of $Q$ have labels $(x+y-1,x+y)$ and the single child node of $Q'$ has label $x+y+1$.
\item $Q$ is a $v$-node and $Q'$ is a $u$-node. Then $P$ and $P'$ are consecutive nodes with labels $x$ and $x+1$ ($P$ is a $u$-node and $P'$ is a $v$-node). The single child node of $Q$ has label $x+y$ and the two child nodes of $Q'$ have labels $((x+1)+y,(x+1)+(y+1)) = (x+y+1,x+y+2)$. 
\end{enumerate}
In all cases, the labels at level $n$ are consecutive integers.  
\end{proof}

\begin{corollary}\label{corollary:label:node}
The label $y$ of a node at level $n$ of the tree is given by the Wythoff sequences: $y = A_n-1 + u(k)$ if the node is a $u$-node, where $k$ is the number of $u$-nodes to the left of that node, and $y = A_n-1 + v(l)$ if the node is a $v$-node, where $l$ is the number of $u$-nodes to the left of the parent $u$-node of that node.
\end{corollary}
\begin{proof}
The pattern of $u$-nodes and $v$-nodes nodes at level $n$ is described by the Fibonacci word $W_n$ (Proposition \ref{proposition:fibonacci:tree:structure}). According to Theorem \ref{theorem:label:Fab}, the label of a node at this level is equal to its position plus the offset $A_n - 1$. Proposition \ref{proposition:fibonacci:word:u} now gives the result. 
\end{proof}

\begin{corollary}\label{corollary:label:parent}
Let $Q$ be a node at level $n$. The label of the parent node of $Q$ is $x = A_{n-1}-1 + k$ where $k$ is the number of $u$-nodes to the left of $Q$.
\end{corollary}
\begin{proof}
The label of the parent node at level $n-1$ is equal to its position plus the offset $A_{n-1} - 1$, and we use Proposition \ref{proposition:fibonacci:word:parent}.
\end{proof}

\begin{example}\label{example:label}
In the tree $\mathcal{F}^{0,1}$ (Fig. \ref{figure:fibtree:F01}), consider the $u$-node $Q$ with label $y = -2$ at level 5. The leftmost label is $-7$ and there are $k = 4$ $u$-nodes to the left of $Q$. We check that $y = -7 - 1 + u(4) = -8 + 6 = -2$. The parent node $P$ of $Q$ has label $x = -4 - 1 + 4 = -1$. The child $v$-node $R$ of $Q$ has label $z = -12 - 1 + v(4) = -13 + 10 = -3$.  
\end{example}

\section{The Hofstadter tree}\label{section:wythoff}

In this section, the labeled tree $\mathcal{F}^{1,2}$ is considered (Fig. \ref{figure:fibtree:F12}). For this tree, the sequences of leftmost and rightmost labels are
\begin{displaymath}
\begin{array}{l l}
\mathbf{A}: &1, 1, 1, 1, \ldots,\\
\mathbf{B}: &1, 2, 3, 5, 8, \ldots, \quad B_n = F_{n+2}.
\end{array}
\end{displaymath}

\begin{proposition}\label{proposition:label:F12}
At level $n$ of the tree $\mathcal{F}^{1,2}$, the labeling is $\llbracket 1 \cdots F_{n+2} \rrbracket$.
\end{proposition}

\begin{proof}
This is immediate from Theorem \ref{theorem:label:Fab}.
\end{proof}

In the tree $\mathcal{F}^{1,2}$, the subtree at the right of the root (Fig. \ref{figure:fibtree:F12}) is the \textit{Hofstadter tree}. The main ascending branch of the Hofstadter tree is the Fibonacci sequence $\mathbf{F}^{1,2} \sim \mathbf{F}^{0,1}$, the second ascending branch is the Lucas sequence $4, 7, 11, 18, \ldots \sim \mathbf{F}^{2,1}$.

\begin{figure}[ht]
\centering
\includegraphics[scale = 0.44]{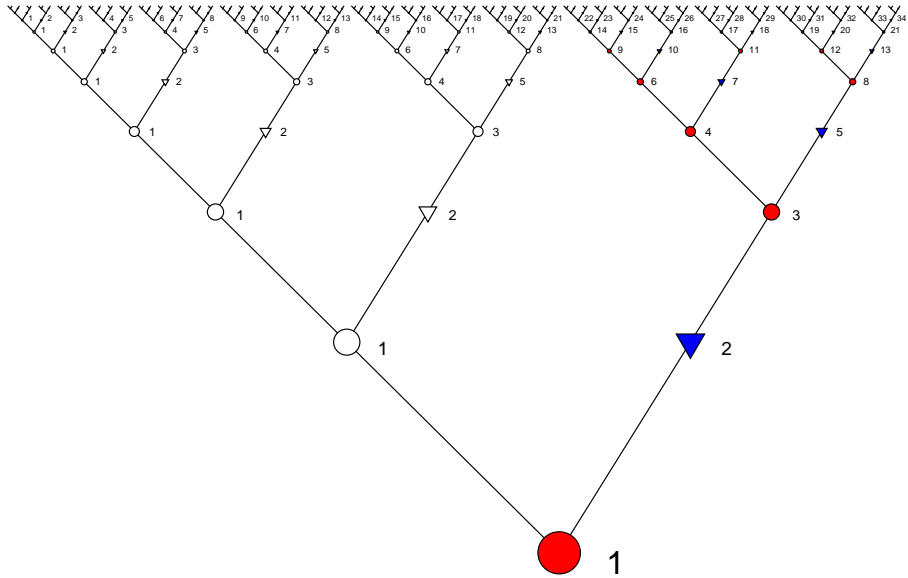}
\caption{The labeled Fibonacci tree $\mathcal{F}^{1,2}$ is made of infinitely many juxtaposed copies of the Hofstadter tree, highlighted on the right.}
\label{figure:fibtree:F12}
\end{figure}

\begin{theorem}\label{theorem:label:Wtree}
Reading the labels of the Hofstadter tree from the root produces the sequence of positive integers.
\end{theorem}
\begin{proof}
Using Proposition \ref{proposition:label:F12}, this is clear from Figure \ref{figure:fibtree:proof3}.
\end{proof}

Tognetti et al. \cite{Tognetti 1990} prove the result of Theorem \ref{theorem:label:Wtree} the other way: they first label the Hofstadter tree by consecutive integers from the root, then they show that the labeling is consistent with the Fibonacci generation scheme. 

\begin{figure}[ht]
\centering
\includegraphics[scale = 0.44]{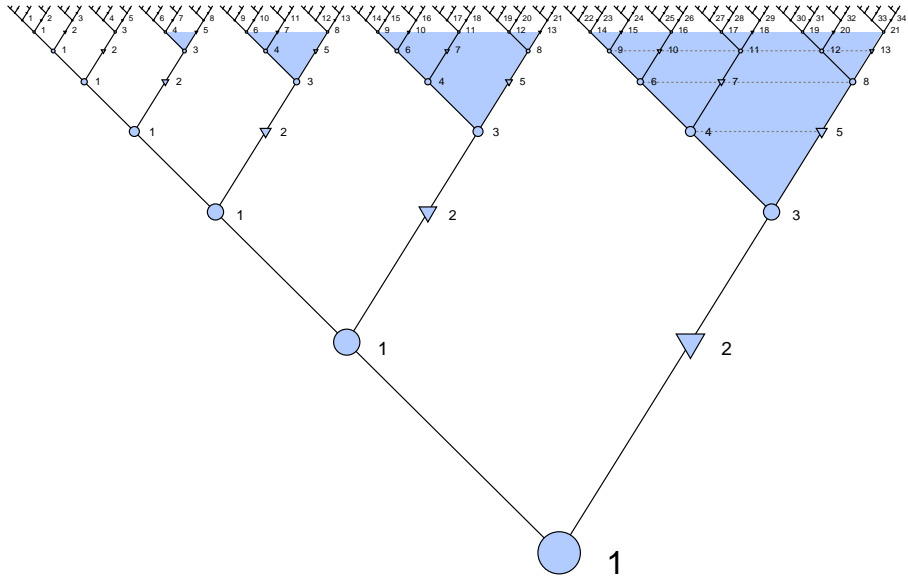}
\caption{Illustration for the proof of Theorem \ref{theorem:label:Wtree}.}
\label{figure:fibtree:proof3}
\end{figure}

\begin{definition}\label{definition:primitive}
In the tree $\mathcal{F}^{a,b}$, a $u$-node is \textit{primitive} if its parent node is a $u$-node. Let this primitive node and its parent node have labels $y$ and $x$ respectively. Then the child $v$-node of the primitive node has label $x+y$, and $(y,x+y)$ is called a \textit{primitive tree-pair}. This pair starts a Fibonacci sequence along the ascending branch rooted at the primitive node. 
\end{definition}

According to the definition, the root node of the tree is a $u$-node that is not primitive.

\begin{proposition}\label{proposition:wythoffarray}
The ascending branches of the Hofstadter tree read from the root form the successive rows of the Wythoff array. 
\end{proposition}
\begin{proof}
By definition, the first two columns of the Wythoff array contain the primitive Wythoff pairs $\{(uu(i),vu(i)); i \in \mathbb{N}^*\}$ \cite{Silber 1977}, and these pairs are extended by the Fibonacci recursion to form the rows \cite{Morrison 1980} (sequence A035513 in the OEIS \cite{Sloane}). The first Wythoff pair $(1,2)$ corresponds to the first two labels on the main ascending branch of the Hofstadter tree, leading to the Fibonacci sequence. By Proposition \ref{proposition:label:F12}, the label of a given node at level $n$ of the whole tree $\mathcal{F}^{1,2}$ corresponds to the position of the corresponding letter in the Fibonacci word $W_n$ ($W_n$ describes the pattern of $u$-nodes and $v$-nodes by Proposition \ref{proposition:fibonacci:tree:structure}). Let us consider a primitive node on the Hofstadter tree. Its parent node is by definition a $u$-node, whose label is $x = u(i)$ for some $i > 1$ by Proposition \ref{proposition:fibonacci:word:u}. By Proposition \ref{proposition:fibonacci:word:u} again, the primitive node has label $y = u(k)$ where $k$ is the number of $u$-nodes to the left of the node. By Proposition \ref{proposition:fibonacci:word:parent}, $x = k$. Then $y = uu(i)$ is the first term of a primitive Wythoff pair and, by (\ref{equation:v}), $x+y = u(i)+uu(i) = vu(i)$ is the second term of the pair. 
\end{proof}

Structural properties of the Wythoff array \cite{Kimberling 1993} can be read from its representation as the Hofstadter tree.

Theorem \ref{theorem:label:Wtree} recovers the fact that the Wythoff array contains every positive integer exactly once. In section \ref{section:representation}, we shall also recover the fact that the Wythoff array represents all positive Fibonacci sequences, in the sense that any positive Fibonacci sequence is equivalent to a sequence in the array \cite{Morrison 1980}. 

\section{The set of labeled Fibonacci trees}\label{section:set}

On the set of labeled Fibonacci trees, 
\begin{displaymath}
\Phi = \{\mathcal{F}^{a,b};\; (a,b) \in \mathbb{Z}^2\},
\end{displaymath}
the sum $\mathcal{F} \oplus \mathcal{F'}$ of two trees $\mathcal{F}$ and $\mathcal{F'}$ is defined as the labeled tree obtained by superimposing the two trees and adding the labels of the corresponding nodes with a correction term: 
the labeling of the sum at level $n$ is defined by
\begin{displaymath}
\mathcal{L}_n(\mathcal{F} \oplus \mathcal{F'}) = \mathcal{L}_n(\mathcal{F}) + \mathcal{L}_n(\mathcal{F'}) - \mathcal{L}^{0,0}_n,
\end{displaymath}
where $\mathcal{L}^{0,0}_n$ is the interval $\llbracket -F_{n+2} + 1 \cdots 0 \rrbracket$.

\begin{lemma}\label{lemma:converse}
If the Fibonacci tree is labelled at each level $n$ by the interval $\mathcal{L}^{a,b}_n$ then the tree is $\mathcal{F}^{a,b}$. 
\end{lemma}
\begin{proof}
We take a Fibonacci tree, label the root $\mathcal{L}^{a,b}_0 = \{a\}$ and the first level $\mathcal{L}^{a,b}_1 = \{b-1,b\}$, and then apply the labeling rules. The tree obtained is $\mathcal{F}^{a,b}$, and it has labeling $\mathcal{L}^{a,b}_n$ at each level $n$ by Theorem \ref{theorem:label:Fab}. Hence the two procedures -- labeling by $\mathcal{L}^{a,b}_n$ at each level $n$ or labeling according to the rules -- lead to the same labeled tree.  
\end{proof}

\begin{theorem}\label{theorem:treesum}
The set $(\Phi,\oplus)$ of labeled Fibonacci trees is a commutative group isomorphic to $(\mathbb{Z}^2,+)$:
\begin{displaymath}
a,b,a',b' \in \mathbb{Z}, \quad \mathcal{F}^{a,b} \oplus \mathcal{F}^{a',b'} = \mathcal{F}^{a+a',b+b'}.
\end{displaymath}
The identity element is the tree $\mathcal{F}^{0,0}$.
\end{theorem}
\begin{proof}
We have to show that given $\mathcal{F} = \mathcal{F}^{a,b}$ and $\mathcal{F'} = \mathcal{F}^{a',b'}$, the labeled Fibonacci tree $\mathcal{F} \oplus \mathcal{F'}$ is an element of $\Phi$, which is $\mathcal{F}^{a+a',b+b'}$. Using (\ref{equation:Fab}), we have $F^{a,b}_n + F^{a',b'}_n = F^{a+a',b+b'}_n$, so that the labeling of $\mathcal{F} \oplus \mathcal{F'}$ at level $n$ is
\begin{displaymath}
\mathcal{L}^{0,0}_n + F^{a,b}_n + \mathcal{L}^{0,0}_n + F^{a',b'}_n - \mathcal{L}^{0,0}_n = \mathcal{L}^{0,0}_n + F^{a+a',b+b'}_n = \mathcal{L}^{a+a',b+b'}_n. 
\end{displaymath}
By Lemma \ref{lemma:converse}, an element of $\Phi$ is entirely determined by the labelings $\mathcal{L}^{a,b}_n$, concluding the proof.
\end{proof}

For $\lambda \in \mathbb{Z}$, define
\begin{displaymath}
\lambda \mathcal{F}^{a,b} = \mathcal{F}^{\lambda a,\lambda b}.
\end{displaymath}
The notation $\lambda \mathcal{F}^{a,b}$ means the sum of $\lambda$ copies of the tree $\mathcal{F}^{a,b}$ for $\lambda \ge 0$, and the sum of $-\lambda$ copies of the tree $\mathcal{F}^{-a,-b}$ for $\lambda < 0$. We obtain a generalization of (\ref{equation:Fab}):
\begin{displaymath}
\mathcal{F}^{a,b} = a\mathcal{F}^{1,0} \oplus b\mathcal{F}^{0,1}.
\end{displaymath}

\section{Representation properties}\label{section:representation}

The labeled Fibonacci tree $\mathcal{F}^{a,b}$ \textit{represents} $\mathbb{Z}$ if every interval of $\mathbb{Z}$ is contained in the labeling of $\mathcal{F}^{a,b}$ at some level (and therefore at all higher levels). $\mathbb{Z}_-$ denotes the set of nonpositive integers, $\mathbb{Z}_+$ denotes the set of positive integers.

\begin{proposition}\label{proposition:representZ}
The tree $\mathcal{F}^{a,b}$ represents $\mathbb{Z}$ if and only if $ \; 0 < a + b\varphi < \varphi^3$.
\end{proposition}
\begin{proof}
The leftmost label at level $n$ of $\mathcal{F}^{a,b}$ is $F^{a,b}_n - F_{n+2} + 1$. Using (\ref{equation:Fab}),
\begin{displaymath}
F^{a,b}_n - F_{n+2} = aF_{n-1} + bF_{n} - F_{n-1} - 2F_n = (a-1)F_{n-1} + (b-2)F_{n} = F^{a-1,b-2}_n.
\end{displaymath}
Therefore, the integer interval represented by $\mathcal{F}^{a,b}$ at level $n$ is
\begin{displaymath}
\mathcal{L}^{a,b}_n = \llbracket F^{a-1,b-2}_n + 1 \cdots F^{a,b}_n \rrbracket. 
\end{displaymath}
If $\mathcal{F}^{a,b}$ represents $\mathbb{Z}$, we have 
\begin{displaymath}
(a-1)F_{n-1} + (b-2)F_n + 1 < 0, \quad aF_{n-1} + bF_n > 0, 
\end{displaymath}
for large $n$. The first inequality gives
\begin{displaymath}
(a-1) + (b-2)\frac{F_n}{F_{n-1}} + \frac{1}{F_{n-1}} < 0.
\end{displaymath}
When $n \rightarrow +\infty$, as $\frac{F_n}{F_{n-1}} \rightarrow \varphi$, we obtain $a + b\varphi - (1 + 2\varphi) \leqslant 0$. So, $a + b\varphi \leqslant 1 + 2\varphi = \varphi^3$, with equality only if $(a,b) = (1,2)$. But we know from Figure \ref{figure:fibtree:F12} that the tree $\mathcal{F}^{1,2}$ does not represent $\mathbb{Z}$. Thus, $a + b\varphi < \varphi^3$. The second inequality gives 
\begin{displaymath}
a + b\frac{F_n}{F_{n-1}} > 0.
\end{displaymath}
When $n \rightarrow +\infty$, we obtain $a + b\varphi \geqslant 0$ with equality only if $(a,b) = (0,0)$. But Figure \ref{figure:fibtree:F01} shows that the tree $\mathcal{F}^{0,0}$ does not represent $\mathbb{Z}$. Thus, $a + b\varphi > 0$. 

Conversely, assume $a + b\varphi < \varphi^3$. Then there exists $\epsilon > 0$ such that
\begin{displaymath}
(a-1) + (b-2)\varphi + \epsilon \leqslant 0. 
\end{displaymath}
The relation
\begin{displaymath}
\varphi = \frac{F_n}{F_{n-1}} - \frac{1}{F_{n-1}}\frac{(-1)^{n-1}}{\varphi^{n-1}}
\end{displaymath}
gives
\begin{displaymath}
(a-1)F_{n-1} + (b-2)F_n + (b-2)\frac{(-1)^{n-1}}{\varphi^{n-1}} \leqslant -\epsilon F_{n-1} < 0.
\end{displaymath}
This shows that when $n \rightarrow +\infty$, $(a-1)F_{n-1} + (b-2)F_n \rightarrow -\infty$. Similarly, the condition $a + b\varphi > 0$ leads to $aF_{n-1} + bF_n \rightarrow +\infty$ when $n \rightarrow +\infty$.
\end{proof}

The infinite set of pairs $(a,b)$ satisfying the conditions of Proposition \ref{proposition:representZ} is depicted in Figure \ref{figure:lattice}. The tree $\mathcal{F}^{0,0}$ is the only one representing $\mathbb{Z_-}$ exactly, and the tree $\mathcal{F}^{1,2}$ is the only one representing $\mathbb{Z_+}$ exactly (open circles on Fig. \ref{figure:lattice}). 

\begin{figure}[ht]
\centering
\includegraphics[scale = 0.44]{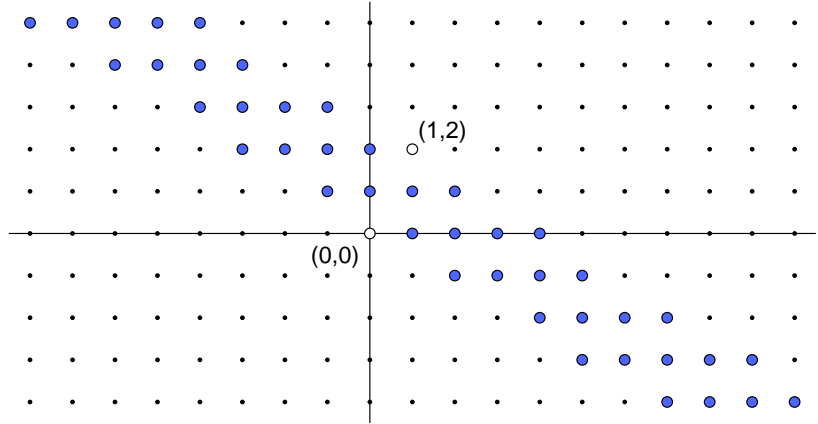}
\caption{Filled circles correspond to the pairs $(a,b)$ for which the tree $\mathcal{F}^{a,b}$ represents every interval of $\mathbb{Z}$.}
\label{figure:lattice}
\end{figure}

The set of Fibonacci sequences is denoted
\begin{displaymath}
\mathbb{F} = \{\mathbf{F}^{a,b}; (a,b) \in \mathbb{Z}^2\}. 
\end{displaymath}
The set of nonpositive Fibonacci sequences, whose terms eventually belong to $\mathbb{Z}_-$, is denoted $\mathbb{F}_-$, and the set of positive Fibonacci sequences, whose terms eventually belong to $\mathbb{Z}_+$, is denoted $\mathbb{F}_+$.

The labeled tree $\mathcal{F}^{a,b}$ \textit{represents} an element of $\mathbb{F}$ if the sequence is equivalent (identical up to a shift of index) to a sequence appearing along an ascending branch of the tree. $\mathcal{F}^{a,b}$ \textit{represents} $\mathbb{F}$ if it represents any element of $\mathbb{F}$.

The set of trees representing $\mathbb{Z}$ (Proposition \ref{proposition:representZ}) is denoted
\begin{displaymath}
\Psi = \{ \mathcal{F}^{a,b}; (a,b) \in \mathbb{Z}^2, \; 0 < a + b\varphi < \varphi^3 \}.
\end{displaymath}  

\begin{theorem}\label{theorem:representF} 
A labeled Fibonacci tree represents the set of Fibonacci sequences if and only if it is an element of $\Psi$.
\end{theorem}
\begin{proof}
We take a tree $\mathcal{F}^{a,b} \in \Psi$ and a given Fibonacci sequence, and show that that there exists a pair of consecutive terms of the sequence that appears as consecutive labels along an ascending branch of $\mathcal{F}^{a,b}$. Like for the Wythoff array, the argument is based on the fact that pairs of consecutive terms of any Fibonacci sequence eventually become Wythoff pairs \cite{Morrison 1980,Vandervelde 2012}. In fact, we show that every primitive Wythoff pair $(uu(j),vu(j))$, $j \in \mathbb{Z}^*$, appears as a primitive tree-pair (see Definition \ref{definition:primitive} in Section \ref{section:wythoff}). The zero sequence $\mathbf{F}^{0,0}$ is represented by a primitive tree-pair, but has no corresponding Wythoff pair, and is treated separately. 

Let $j \in \mathbb{Z}^*$. We look for a primitive $u$-node $Q$ labeled $y = uu(j)$ and its parent $u$-node $P$ labeled $x = u(j)$. Then $x + y = u(j) + uu(j) = vu(j)$, so that the primitive tree-pair $(y,x+y)$ represents the given Fibonacci sequence. Let us denote $G_n = F^{a-1,b-2}_n$. The leftmost label $A_n$ at level $n$ satisfies $A_n - 1 = G_n$. As the tree represents every interval of $\mathbb{Z}$ (Proposition \ref{proposition:representZ}), we can find $n_0 \geqslant 2$ such that $n \geqslant n_0$ implies $j \in \mathcal{L}^{a,b}_{n-2}$, i.e., $1 + G_{n-2} \leqslant j \leqslant G_{n-2} + F_n$. Hence, if we set $i_n = j - G_{n-2}$, then $1 \leqslant i_n \leqslant F_n$. By Lemma \ref{lemma:uG}, there exists $n_1$ such that for any $n \geqslant n_1$ we have $u(i_n + G_{n-2}) = u(i_n) + G_{n-1}$. We take $n$ larger than $n_0$ and $n_1$ and set $i = i_n$. The node $P$ with label $x = G_{n-1} + u(i)$ is a $u$-node, by Corollary \ref{corollary:label:node}. (The condition $1 \leqslant i \leqslant F_n$ ensures that $x$ is the label of a node at level $n-1$. Indeed, according to Corollary \ref{corollary:label:node}, $i$ is the number of $u$-nodes to the left of $P$. This number must be at least 1, and at most $F_n$, the total number of $u$-nodes at level $n-1$.) By Corollary \ref{corollary:label:parent}, $x = G_{n-1} + k$, where $k$ is the number of $u$-nodes to the left of the $u$-node $Q$. Thus $k = u(i)$. By Corollary \ref{corollary:label:node} again, $Q$ has label $y = G_n + u(k) = G_n + uu(i)$. $Q$ is a primitive node. We now have $x = G_{n-1} + u(i) = u(i + G_{n-2}) = u(j)$, and 
\begin{displaymath}
v(j) = j + u(j) = i + G_{n-2} + G_{n-1} + u(i) = G_n + i + u(i) = G_n + v(i).
\end{displaymath}
Using (\ref{equation:uu}), we obtain $uu(j) = G_n + uu(i) = y$. 

For the sequence $\mathbf{F}^{0,0}$, we proceed similarly. By Lemma \ref{lemma:uG0}, there exists a unique $n \geqslant \nu+2$ such that $x = G_{n-1} + u(i) = 0$ with $i = 1 - G_{n-2}$, and where the reference index $\nu$ of $\mathbf{G}$ has the property that $G_{\nu}$ is the largest nonpositive term of the sequence. As in the general case, we consider the $u$-node $P$ at level $n-1$ whose label is $x = 0$, and its child $u$-node $Q$, whose label is $y = G_n + uu(i)$. We now have, using (\ref{equation:uu}), 
\begin{displaymath}
y = G_n + uu(i) = G_n + i + u(i) - 1 = G_n + 1 - G_{n-2} - G_{n-1} - 1 = 0.   
\end{displaymath}
Thus, the primitive tree-pair $(y,x+y) = (0,0)$ represents $\mathbf{F}^{0,0}$. To complete this part of the proof, it must be checked that $i_n = i$ satisfies $1 \leqslant i_n \leqslant F_n$. By assumption, there exists a smallest $n_0 \in \mathbb{N}$ such that for $m \geqslant n_0+2$, the integer $1$ is represented at level $m-2$, i.e., $1 + G_{m-2} \leqslant 1 \leqslant G_{m-2} + F_m$. Hence, if $i_m = 1 - G_{m-2}$, then $1 \leqslant i_m \leqslant F_m$ for $m \geqslant n_0+2$. We prove that $\nu \geqslant n_0$. As $n \geqslant \nu+2$, this will imply $n \geqslant n_0+2$, and $1 \leqslant i_n \leqslant F_n$ as desired. By definition of $\nu$, $G_{\nu} \leqslant 0$, i.e., $1 + G_{\nu} \leqslant 1$,  and it remains to show that $G_{\nu} + F_{\nu+2} \geqslant 1$. If, on the contrary, $G_{\nu} + F_{\nu+2} < 1 $, then $G_{\nu} + F_{\nu+2} \leqslant 0$. But, by definition of $\nu$, $G_{\nu} > G_{\nu-1}$. This gives
\begin{displaymath}
G_{\nu-1} + F_{\nu+1} < G_{\nu-1} + F_{\nu+2} < G_{\nu} + F_{\nu+2} \leqslant 0.   
\end{displaymath}
By the Fibonacci recursion, $G_{\nu-1} + F_{\nu+1} \leqslant 0$ and $G_{\nu} + F_{\nu+2} \leqslant 0$ lead to $G_{\nu+1} + F_{\nu+3} \leqslant 0$. We can pursue the recursion to get $G_{\nu+p} + F_{\nu+2+p} \leqslant 0$ for any $p \in \mathbb{N}$. This contradicts the fact that $F^{a,b}_m = G_{m} + F_{m+2} \rightarrow \infty$.

The reasoning of the previous paragraph does not work for the tree $\mathcal{F}^{0,1}$ because the tree-pair $(0,0)$ appears at level 1 (in fact, $\nu = -1$). It is the only exceptional case. Nevertheless, the formulas still hold, and it can also be seen directly that $\mathbf{F}^{0,0}$ is represented by $\mathcal{F}^{0,1}$ (Fig. \ref{figure:fibtree:F01}).

To show the converse in the theorem, we note that if a tree $\mathcal{F}^{a,b}$ is not an element of $\Psi$, it does not represent $\mathbb{Z}$, and there are Fibonacci sequences that are not represented by the tree.
\end{proof}

When $\mathcal{F}^{a,b} \in \Psi$, the pair $(a,b)$ is not a Wythoff pair. Indeed, its terms are either of opposite sign, or in the 8 cases where the terms are nonnegative (Fig. \ref{figure:lattice}), they do not form a Wythoff pair. The primitive Wythoff pair $(c,d)$ corresponding to the sequence $\mathbf{F}^{a,b}$ is positive, and appears further up in the main branch of the tree : there exists $n > 0$ such that $F^{a,b}_n = c$ and $F^{a,b}_{n+1} = d$. The pair $(c,d)$ also appears elsewhere in the tree, as any primitive Wythoff pair. For example, in the tree $\mathcal{F}^{0,1} \in \Psi$, the Wythoff pair $(1,2)$ representing the Fibonacci sequence $\mathbf{F} = \mathbf{F}^{0,1}$ appears at level 2 as a nonprimitive tree-pair, and appears at all levels $n \geqslant 3$ as a primitive tree-pair (Fig. \ref{figure:fibtree:F01}).

\begin{proposition}\label{proposition:pairs} 
In a tree $\mathcal{F} \in \Psi$, every Fibonacci sequence is represented by infinitely many branches, except for the zero sequence $\mathbf{F}^{0,0}$ that is represented by a single branch.
\end{proposition}
\begin{proof}
The pair $(0,0)$ appears as a unique primitive tree-pair, as seen in the proof of Theorem \ref{theorem:representF}. Therefore, the sequence $\mathbf{F}^{0,0}$ is represented by a single ascending branch of the tree. Also from the proof of Theorem \ref{theorem:representF}, a primitive Wythoff pair appears in the tree as a primitive tree-pair at all levels above some level. Hence it appears in infinitely many different ascending branches, since each primitive tree-pair is rooted at a primitive node that starts a new branch.
\end{proof}

To summarize, the set $\Phi$ can be partionned into three subsets: (1) the trees $\mathcal{F}^{a,b}$ such that $0 < a + b\varphi < \varphi^3$, constituting $\Psi$, and representing $\mathbb{Z}$ and $\mathbb{F}$, (2) those such that $a + b\varphi \leqslant 0$ representing subsets of $\mathbb{Z}_-$ and $\mathbb{F}_-$, and (3) those such that $a + b\varphi \geqslant \varphi^3$, representing subsets of $\mathbb{Z}_+$ and $\mathbb{F}_+$.

\section{Order relation}\label{section:order}

For two trees $\mathcal{F}, \mathcal{F'} \in \Phi$, the notation $\mathcal{F}  \vartriangleleft \mathcal{F}'$ means that $\mathcal{F}$ is a subtree of $\mathcal{F}'$ such that the root of $\mathcal{F}$ is a $u$-node of $\mathcal{F}'$, and corresponding labels are identical. When $\mathcal{F} \vartriangleleft \mathcal{F}'$, we say that $\mathcal{F}'$ \textit{contains} $\mathcal{F}$. For example, the tree $\mathcal{F}^{0,1}$ contains $\mathcal{F}^{0,0}$ and $\mathcal{F}^{1,2}$ as subtrees (Fig. \ref{figure:fibtree:F01}).\\

We shall use two affine maps on the ring $\mathbb{Z}[\varphi]$:
\begin{displaymath}
L(x + y\varphi) = (y-1) + (x+y-1)\varphi, \quad R(x + y\varphi) = (x+y) + (x + 2y)\varphi.
\end{displaymath} 
The map $L$ sends the pair $(x,y)$ at the root of the tree $\mathcal{F}^{x,y}$ to the pair $(y-1,x+y-1)$ at the root of the first left subtree (highlighted on the left of Fig. \ref{figure:fibtree:recursion}). The map $R$ sends the pair $(x,y)$ to the pair $(x+y,x+2y)$ at the root of the first right subtree (highlighted on the right of Fig. \ref{figure:fibtree:recursion}).
The relations
\begin{displaymath}
(y-1)  + (x+y-1)\varphi = \varphi(x + y\varphi) - \varphi^2, \quad
(x + y) + (x+2y)\varphi = \varphi^{2}(x + y\varphi),
\end{displaymath}
show that for $z = x + y\varphi \in \mathbb{Z}[\varphi]$, 
\begin{displaymath}
L(z) = \varphi z - \varphi^2 = \varphi(z - \varphi^3) + \varphi^3, \quad R(z) = \varphi^{2}z.
\end{displaymath}

Underlying this formulation are the group isomorphisms:
\begin{displaymath}
\begin{array}{c c c c c}
\Phi &\overset\sim\longrightarrow& \mathbb{Z}^2 &\overset\sim\longrightarrow& \mathbb{Z}[\varphi]\\
\mathcal{F}^{a,b} &\longmapsto& (a,b) &\longmapsto& a + b\varphi.
\end{array}
\end{displaymath}

\begin{proposition}\label{proposition:selfcontain} 
The trees $\mathcal{F}^{1,2}$ and $\mathcal{F}^{0,0}$ are the only elements of $\Phi$ containing nested copies of themselves.
\end{proposition}
\begin{proof}
It is clear that the trees $\mathcal{F}^{1,2}$ and $\mathcal{F}^{0,0}$ contain themselves infinitely many times as proper subtrees (see Fig. \ref{figure:fibtree:F01}). Conversely, assume that the tree $\mathcal{F}^{a,b}$ contains itself as a proper subtree. Then it contains a $u$-node labeled $a$ that is not the root node, and whose child $v$-node is labeled $b$. This pair $(a,b)$ up in the tree can be reached from the root pair $(a,b)$ by applying the maps $R$ and $L$ to $a + b\varphi$ in $\mathbb{Z}[\varphi]$. In other words, $z^* = a + b \varphi$ is a fixed point of a composition of $L$ and $R$, and we have to solve $z = L^{p_1}R^{q_1} \cdots L^{p_k}R^{q_k}z$ or $z = R^{q_1}L^{p_1} \cdots R^{q_k}L^{p_k}z$, with $p_i, q_i \geqslant 0$ not all zero. The iterates of the maps $L$ and $R$ are 
\begin{displaymath}
L^p(z) = \varphi^{p}(z - \varphi^3) + \varphi^{3}, \quad R^{q}(z) = \varphi^{2q}z.
\end{displaymath}
The unique fixed point of $L^p$ is $z^* = \varphi^3 = 1 + 2\varphi$. The unique fixed point of $R^q$ is $z^* = 0$. 

We show that either $z = L^p z$, in which case the fixed point is $z^* = 1 + 2\varphi$ leading to $(a,b) = (1,2)$, or $z = R^q z$, in which case the fixed point is $z^* = 0$ leading to $(a,b) = (0,0)$. Using the formula
\begin{displaymath}
(L^pR^q - R^qL^p)z = \varphi^3(\varphi^p - 1)(\varphi^{2q} - 1) = \xi_{p,q},
\end{displaymath}
we obtain
\begin{displaymath}
L^{p_1}R^{q_1} \cdots L^{p_k}R^{q_k}z = R^{q_1 + \cdots + q_k}L^{p_1 + \cdots + p_k}z + \xi_{p_1,q_1} + \cdots + \xi_{p_k,q_k}.
\end{displaymath}
With $p = p_1 + \cdots + p_k$, $q = q_1 + \cdots + q_k$, the equation for the fixed point becomes
\begin{displaymath}
z = R^{q}L^{p}z + \xi_{p_1,q_1} + \cdots + \xi_{p_k,q_k}.
\end{displaymath}
Setting $z = \varphi^3$ leads to
\begin{displaymath}
1 = \varphi^{2q} + (\varphi^{p_1} - 1)(\varphi^{2q_1} - 1) + \cdots + (\varphi^{p_k} - 1)(\varphi^{2q_k} - 1).
\end{displaymath}
The only solution is $q_i = 0$ for all $i$, giving $z = L^p z$. Similarly,
\begin{displaymath}
R^{q_1}L^{p_1} \cdots R^{q_k}L^{p_k}z = L^{p_1 + \cdots + p_k}R^{q_1 + \cdots + q_k}z - (\xi_{p_1,q_1} + \cdots + \xi_{p_k,q_k})
\end{displaymath}
gives
\begin{displaymath}
z = L^{p}R^{q}z - (\xi_{p_1,q_1} + \cdots + \xi_{p_k,q_k}).
\end{displaymath}
Setting $z = 0$ leads to
\begin{displaymath}
1 = \varphi^{p} + (\varphi^{p_1} - 1)(\varphi^{2q_1} - 1) + \cdots + (\varphi^{p_k} - 1)(\varphi^{2q_k} - 1).
\end{displaymath}
The only solution is $p_i = 0$ for all $i$, giving $z = R^q z$.
\end{proof}

\begin{theorem}\label{theorem:partialorder}
The relation $\vartriangleleft$ is a partial order on $\Phi$.
\end{theorem}
\begin{proof}
To prove that $\vartriangleleft$ defines an order relation, only antisymmetry needs to be checked. Assume that for $\mathcal{F},\mathcal{F'} \in \Phi$ we have  $\mathcal{F} \vartriangleleft \mathcal{F'}$ and $\mathcal{F'} \vartriangleleft \mathcal{F}$, but $\mathcal{F} \neq \mathcal{F'}$. Then $\mathcal{F} \vartriangleleft \mathcal{F'} \vartriangleleft \mathcal{F}$, so that $\mathcal{F}$ contains itself as a proper subtree. By Proposition \ref{proposition:selfcontain}, $\mathcal{F} = \mathcal{F}^{0,0}$ or $\mathcal{F} = \mathcal{F}^{1,2}$. If $\mathcal{F} = \mathcal{F}^{0,0}$, the inclusions $\mathcal{F}^{0,0} \vartriangleleft \mathcal{F'} \vartriangleleft \mathcal{F}^{0,0}$ imply that $\mathcal{F'}$ represents $\mathbb{Z}_-$ exactly, and must be $\mathcal{F}^{0,0}$, which is a contradiction. Similarly, $\mathcal{F} = \mathcal{F}^{1,2}$ leads to a contradiction. Thus, $\mathcal{F} = \mathcal{F'}$.
\end{proof}

The order relation $\vartriangleleft$ is not compatible with the group structure. Otherwise, for any trees $\mathcal{F}$, $\mathcal{G}$, $\mathcal{F'}$, $\mathcal{G'}$ we would have
\begin{displaymath}
\mathcal{F} \vartriangleleft \mathcal{G},\; \mathcal{F'} \vartriangleleft \mathcal{G'} \Longrightarrow (\mathcal{F} \oplus \mathcal{F'}) \vartriangleleft (\mathcal{G} \oplus \mathcal{G'}).
\end{displaymath}
A counterexample is given by $\mathcal{F}^{0,0} \vartriangleleft \mathcal{F}^{0,1}$, $\mathcal{F}^{0,0} \vartriangleleft \mathcal{F}^{1,1}$. The tree $\mathcal{F}^{0,0} \oplus \mathcal{F}^{0,0} = \mathcal{F}^{0,0}$ is not a subtree of $\mathcal{F}^{0,1} \oplus \mathcal{F}^{1,1} = \mathcal{F}^{1,2}$. Indeed, the labels of the tree $\mathcal{F}^{1,2}$ are all positive.

We conjecture that any two elements $\mathcal{F}$, $\mathcal{F'}$ of $\Phi$ endowed with the order relation $\vartriangleleft$ have a least upper bound $\mathcal{F} \vee \mathcal{F'}$. For example, Figure \ref{figure:fibtree:F01} shows that $\mathcal{F}^{0,1} = \mathcal{F}^{0,0} \vee \mathcal{F}^{1,2}$. 

This means that, given $\mathcal{F} = \mathcal{F}^{c,d}$ and $\mathcal{F'} = \mathcal{F}^{c',d'}$, we can find $\mathcal{G} = \mathcal{F}^{a,b}$ such that $\mathcal{F} \vartriangleleft \mathcal{G}$ and $\mathcal{F'} \vartriangleleft \mathcal{G}$. Then, as there are only finitely many subtrees between $\mathcal{G}$ and $\mathcal{F}$, and between $\mathcal{G}$ and $\mathcal{F'}$, a least upper bound can be found for $\mathcal{F}$ and $\mathcal{F'}$. This amounts at finding $a + b\varphi$ that is sent to both $c + d\varphi$ and $c' + d'\varphi$ by some composition of the maps $L$ and $R$. For example, 
\begin{displaymath}
18 - 10\varphi = L^{-1}R^{-2}(-1 + 2\varphi) = R^{-1}L^{-1}(-3 + 5\varphi).
\end{displaymath}
This is the smallest solution: it sends $-1 + 2\varphi$ and $-3 + 5\varphi$ to $18 - 10\varphi$, entailing $\mathcal{F}^{18,-10} = \mathcal{F}^{-1,2} \vee \mathcal{F}^{-3,5}$. However, this approach leads to complicated formulas, and we have not proven the conjecture.

\section{Concluding remarks}
We have described a set $\Phi$ of labeled Fibonacci trees representing Fibonacci sequences which has the structure of a commutative group isomorphic to $\mathbb{Z}^2$. The set $\Phi$ is moreover endowed with a partial order for which we conjecture that any two elements have a least upper bound. An infinite subset $\Psi$ of $\Phi$ was determined to represent every integer interval and every Fibonacci sequence. This corresponds to two key features of the Wythoff array extended to $\mathbb{Z}$ (Vandervelde \cite{Vandervelde 2012}). Accordingly, the labeled trees that belong to $\Psi$ can be considered as generalizations of the Wythoff array. 

The extended Wythoff array contains every integer exactly once, except for $-1$ that appears twice and $0$ that does not appear, and represents every nonzero Fibonacci sequence uniquely. For the elements of $\Psi$, every integer interval appears infinitely many times, and every Fibonacci sequence is represented infinitely many times, except for the zero sequence that is represented only once.

Finally, labeled trees similar to the labeled Fibonacci trees studied here could be constructed for other sets of sequences defined by parameterized recursions, e.g., triangular numbers, sequences of powers of 2, Perrin and Perrin-like numbers, Tribonacci and $k$-bonacci numbers, Pell numbers.

\bigskip
\hrule
\bigskip

\noindent 2010 \textit{Mathematics Subject Classification}: 11B39, 11Y55, 05C05.\\

\noindent \textit{Keywords}: Fibonacci sequence, Fibonacci tree, Wythoff pairs, Wythoff array.

\bigskip
\hrule
\bigskip

\end{document}